\documentclass{article}
\usepackage{amsmath, amsthm,amssymb}
\usepackage{graphicx}
\usepackage{verbatim}
\usepackage{tikz}

\begin{document}
\setlength{\textwidth}{5.75in}

\makeatletter
\def\@citex[#1]#2{\if@filesw\immediate\write\@auxout{\string\citation{#2}}\fi
  \def\@citea{}\@cite{\@for\@citeb:=#2\do
    {\@citea\def\@citea{, }\@ifundefined
       {b@\@citeb}{{\bf ?}\@warning
       {Citation `\@citeb' on page \thepage \space undefined}}%
\hbox{\csname b@\@citeb\endcsname}}}{#1}}
\makeatother


\makeatletter
\def\doublespaced{\baselineskip=\normalbaselineskip	
    \multiply\baselineskip by 150			
    \divide\baselineskip by 100}			
\def\doublespace{\doublespaced}				
\makeatother

\makeatletter
\def\mitspaced{\baselineskip=\normalbaselineskip	
    \multiply\baselineskip by 115			
    \divide\baselineskip by 100}			
\def\mitspace{\mitspaced}
\makeatother

\makeatletter
\def\capamispaced{\baselineskip=\normalbaselineskip	
    \multiply\baselineskip by 105			
    \divide\baselineskip by 100}			
\def\capamispace{\capamispaced}
\makeatother

\makeatletter
\def\singlespaced{\baselineskip=\normalbaselineskip}	
\def\singlespace{\singlespaced}				
\makeatother

\makeatletter
\def\triplespaced{\baselineskip=\normalbaselineskip	
    \multiply\baselineskip by 3}			
\makeatother

\makeatletter
\def\widenspacing{\multiply\baselineskip by 125		
    \divide\baselineskip by 100}			
\def\whitespace{\widenspacing}				
\makeatother

\font\eightmsam=msam8
\font\ninemsam=msam9
\font\tenmsam=msam10
\newtheorem{Theorem}{Theorem}[section]
\newtheorem{Corollary}[Theorem]{Corollary}
\newtheorem{Remark}[Theorem]{Remark}
\newtheorem{Proposition}[Theorem]{Proposition}
\newtheorem{Lemma}[Theorem]{Lemma}
\newtheorem{Claim}[Theorem]{Claim}
\newtheorem{Definition}[Theorem]{Definition}
\newtheorem{Example}[Theorem]{Example}
\newcommand{\BB}{{\ninemsam\char'04}}
\renewcommand{\Box}{\mbox{\BB}\medskip}
\newcommand{\optimal}{$\Theta(n^{1/2})$}
\newcommand{\joptimal}{$\Theta(n^{1/j})$}
\newcommand{\nfourth}{$\Theta(n^{1/4})$}
\newcommand{\nsixth}{$\Theta(n^{1/6})$}
\newcommand{\nfourthlog}{$\Theta(n^{1/4}\log(n))$}
\newcommand{\linear}{$\Theta(n)$}
\newcommand{\logn}{$\Theta(\log(n))$}

\newcommand{\N}{\mathbb{N}}
\newcommand{\Z}{Z}

\newcommand{\echtp}{\simeq^{\mathrm{EC}}}

\setlength{\baselineskip}{0.3in}
\hoffset -0.4in
\doublespace
\title{Remarks on pointed digital homotopy}
\author{Laurence Boxer
         \thanks{
    Department of Computer and Information Sciences,
    Niagara University,
    Niagara University, NY 14109, USA;
    and Department of Computer Science and Engineering,
    State University of New York at Buffalo.
    E-mail: boxer@niagara.edu
    }
\and
{P. Christopher Staecker
\thanks{
 Department of Mathematics,
 Fairfield University,
 Fairfield, CT 06823-5195, USA.
 E-mail: cstaecker@fairfield.edu
}
}
}
\date{ }
\maketitle

\begin{abstract}
We present and explore in detail a pair of digital
images with $c_u$-adjacencies that are homotopic but not pointed homotopic.
For two digital loops $f,g: [0,m]_Z \rightarrow X$
with the same basepoint, we introduce the notion
of {\em tight at the basepoint (TAB)} pointed
homotopy, which is more restrictive than ordinary
pointed homotopy and yields some different
results.

We present a variant form of the digital fundamental group. Based on
what we call {\em eventually constant} loops, this version of the fundamental
group is equivalent to that of ~\cite{Boxer99}, but offers the advantage that
eventually constant maps are often easier to work with than the
trivial extensions that are key to the development of the fundamental
group in ~\cite{Boxer99} and many subsequent papers.

We show that homotopy equivalent digital images have isomorphic
fundamental groups, even when the homotopy equivalence does not preserve the basepoint.  This assertion appeared in~\cite{Boxer05a}, but
there was an error in the proof; here, we correct the error.

Key words and phrases: {\em digital topology,
digital image, digitally continuous function,
homotopy, homotopy equivalence, fundamental group
}
\end{abstract}

\section{Introduction}
Digital topology adapts tools from geometric and algebraic topology to
the study of digital images. In this paper, we consider questions of
pointed homotopy in digital topology.
We give an example showing that homotopy equivalence
between digital images $(X,c_u)$ and $(Y,c_v)$ does not imply
pointed homotopy equivalence between these images.
This example is then used to illustrate a new variant on the pointed homotopy of digital loops.
We present an alternate
version of the digital fundamental group that appears to have
advantages over the version introduced in~\cite{Boxer99}.
We correct the argument of~\cite{Boxer05a} for the assertion that
homotopy equivalent connected digital images $(X,\kappa)$ and $(Y,\lambda)$
have isomorphic fundamental groups $\Pi_1^{\kappa}(X,x_0)$ and
$\Pi_1^{\lambda}(Y,y_0)$.

Much of the material in section~\ref{prelim} is quoted or paraphrased from
~\cite{BoxKar2}.

\section{Preliminaries}
\label{prelim}
\subsection{General Properties}
\label{dig-con}
Let {\bf Z} be the set of integers.
A {\em (binary) digital image} is a pair $(X,\kappa)$, where
$X \subset {\bf Z}^n$ for some positive integer~$n$, and $\kappa$ is
some adjacency relation for the members of $X$.

Adjacency relations commonly used in the study of
digital images in ${\bf Z}^n$ include the following ~\cite{Han}.
For an integer $u$ such that $1 \leq u \leq n$, we define an
adjacency relation as follows.  Points
\[ p \, = \, (p_1,p_2,\ldots,p_n), ~ q \, = \, (q_1,q_2,\ldots, q_n) \]
are $c_u$-adjacent~\cite{Boxer05b} if
\begin{itemize}
\item $p \neq q$, and
\item there are at most $u$ distinct indices $i$ for which
      $|p_i - q_i| \, = \, 1$, and
\item for all indices $i$, if $|p_i - q_i| \neq 1$ then $p_i \, = \, q_i$.
\end{itemize}

We often denote a $c_u$-adjacency in ${\bf Z}^n$ by the number of points that
are $c_u$-adjacent to a given point in ${\bf Z}^n$.  E.g.,
\begin{itemize}
\item in ${\bf Z}^1$, $c_1$-adjacency is 2-adjacency;
\item in ${\bf Z}^2$, $c_1$-adjacency is 4-adjacency and $c_2$-adjacency is 8-adjacency.
\item in ${\bf Z}^3$, $c_1$-adjacency is 6-adjacency, $c_2$-adjacency is 18-adjacency, and $c_3$-adjacency is 26-adjacency.
\end{itemize}

More general adjacency relations appear in~\cite{Herman}. The work in \cite{Staecker} treats digital images as abstract sets of points with arbitrary adjacencies without regard for their embeddings in $\mathbf Z^n$.

\begin{Definition} {\rm \cite{Boxer}}
\label{dig-int}
Let $a, b \in {\bf Z}$, $a < b$.  A {\rm digital interval}
is a set of the form
\[ [a,b]_{{\bf Z}} ~=~ \{z \in {\bf Z} ~|~ a \leq z \leq b\} \]
in which $c_1$-adjacency is assumed. $\Box$
\end{Definition}

The following generalizes an earlier definition of~\cite{Rosenfeld}.

\begin{Definition}
\label{cont-connect}
{\rm \cite{Boxer99}}
Let $(X,\kappa)$ and $(Y,\lambda)$ be digital images.  Then the function
$f: X \rightarrow Y$ is $(\kappa,\lambda)$-continuous if and only if
for every pair of $\kappa-$adjacent points
$x_0, x_1 \in X$, either $f(x_0)$ $= f(x_1)$, or $f(x_0)$ and $f(x_1)$
are $\lambda-$adjacent. $\Box$
\end{Definition}

See also~\cite{Chen94,Chen04}, where similar concepts are named
{\em immersion}, {\em gradually varied operator}, or
{\em gradually varied mapping}.





A {\em path} from $p$ to $q$ in $(X, \kappa)$ is a $(2,\kappa)$-continuous function
$F : [0,m]_{\bf Z} \rightarrow X$ such that $F(0)=p$ and $F(m)=q$. For a given path $F$, we define the reverse path, $F^{-1}: [0,m]_{\bf Z} \rightarrow X$ defined by
$F^{-1}(t)=F(m-t)$. A {\em loop} is a path $F: [0,m]_{\bf Z} \rightarrow X$ such that $F(0)=F(m)$.

\subsection{Digital homotopy}
\label{dig-htpy}
Intuitively, a homotopy between continuous functions $f,g: X \rightarrow Y$
is a continuous deformation of, say, $f$ over a time period
until the result of the deformation coincides with $g$.

\begin{Definition}{\rm (\cite{Boxer99};} see also~{\rm \cite{Khalimsky})}
\label{htpy-2nd-def}
Let $X$ and $Y$ be digital images.
Let $f,g: X \rightarrow Y$ be $(\kappa,\lambda)-$continuous functions
and suppose there is a positive integer $m$ and a function
\[ F: X \times [0,m]_{{\bf Z}} \rightarrow Y \]
such that

\begin{itemize}
\item for all $x \in X$, $F(x,0) = f(x)$ and $F(x,m)$ $= g(x)$;
\item for all $x \in X$, the induced function
      $F_x: [0,m]_{{\bf Z}} \rightarrow Y$ defined by
          \[ F_x(t) ~=~ F(x,t) \mbox{ for all } t \in [0,m]_{{\bf Z}}, \]
          is $(c_1,\lambda)-$continuous;
\item for all $t \in [0,m]_{{\bf Z}}$, the induced function
         $F_t: X \rightarrow Y$ defined by
          \[ F_t(x) ~=~ F(x,t) \mbox{ for all } x \in  X, \]
          is $(\kappa,\lambda)-$continuous.
\end{itemize}
Then $F$ is a {\rm digital $(\kappa,\lambda)-$homotopy between} $f$ and
$g$, and $f$ and $g$ are $(\kappa,\lambda)${\rm -homotopic in} $Y$. If $m=1$, then $f$ and $g$ are homotopic \emph{in 1 step}. 

If, further, there exists $x_0 \in X$ such that $F(x_0, t)= F(x_0,0)$ for
all $t \in [0,m]_{\bf Z}$, we say $F$ is a {\rm pointed homotopy}. If $g$ is a constant function, we say $F$ is a
{\rm nullhomotopy}, and $f$ is {\rm nullhomotopic}.
$\Box$
\end{Definition}

The notation $f~\simeq_{( \kappa, \lambda)} ~g$
indicates that functions $f$ and $g$ are digitally
$(\kappa,\lambda)-$homotopic in $Y$. If $\kappa=\lambda$, we abbreviate this as $f\simeq_\kappa g$. When the adjacencies are understood we simply write $f\simeq g$.
%

Digital homotopy is an equivalence relation among digitally continuous
functions~\cite{Khalimsky,Boxer99}.

Let $H: [0,m]_{\bf Z} \times [0,n]_{\bf Z} \rightarrow X$ 
be a homotopy between paths
$f, g: [0,m]_{\bf Z} \rightarrow X$. We say $H$ {\em holds the endpoints fixed} if
 $f(0)=H(0,t)=g(0)$ and 
$f(m)=H(m,t)=g(m)$ for all
$t \in [0,n]_{\bf Z}$. If $f$ and $g$ are loops,
we say $H$ is {\em loop preserving} if 
$H(0,t)=H(m,t)$ for all $t \in [0,n]_{\bf Z}$.
Notice that if $f$ and $g$ are loops and
$H$ holds the endpoints
fixed, then $H$ is a loop preserving pointed homotopy 
between $f$ and $g$.

As in classical topology, we say two digital images $(X,\kappa)$ and $(Y,\lambda)$ are \emph{homotopy equivalent} when there are continuous functions $f:X\to Y$ and $g:Y\to X$ such that $g\circ f \simeq_{(\kappa,\lambda)} 1_X$ and $f\circ g \simeq_{(\lambda,\kappa)} 1_Y$. 

\subsection{Digital fundamental group}

If $f$ and $g$ are paths in $X$
such that $g$ starts where $f$ ends,
the {\em product} (see~\cite{Khalimsky})
of $f$ and $g$, written $f * g$, is, intuitively, the path obtained
by following $f$, then following $g$.
Formally, if $f: [0,m_1]_{{\bf Z}} \rightarrow X$,
$g: [0,m_2]_{{\bf Z}} \rightarrow X$, and $f(m_1)=g(0)$, then
$(f * g): [0,m_1+m_2]_{{\bf Z}} \rightarrow X$ is defined by
\[(f * g)(t) = \left\{
         \begin{array}{ll}
               f(t) &  \mbox{if } t \in [0,m_1]_{{\bf Z}}; \\
               g(t - m_1) & \mbox{if } t \in [m_1,m_1+m_2]_{{\bf Z}}.
         \end{array}
        \right .  \]
Restriction of loop classes to loops defined on the
same digital interval would be undesirable.  The following
notion of {\em trivial extension} to permit a loop
to ``stretch" within the same pointed homotopy class. In
section~\ref{ec-section}, we will introduce a different method of
``stretching'' a loop within its pointed homotopy class.
Intuitively, $f'$ is a trivial extension of $f$ if $f'$ follows the same
path as $f$, but more slowly, with pauses for rest
(subintervals of the domain on which $f'$ is constant).

\begin{Definition}
\label{triv-extension}
{\rm \cite{Boxer99}}
Let $f$ and $f'$ be loops in a pointed digital image $(X,x_0)$.
We say
$f'$ is a {\rm trivial extension of} $f$ if there are sets of paths
$\{f_1, f_2, \ldots, f_k\}$ and $\{F_1, F_2, \ldots, F_p\}$ in $X$ such that
\begin{enumerate}
  \item $0<k \leq p$;
  \item $f = f_1 * f_2 * \ldots * f_k$;
  \item $f' = F_1 * F_2 * \ldots * F_p$;
  \item there are indices
        $1 \leq i_1 < i_2 < \ldots < i_k \leq p$ such that
  \begin{itemize}
    \item $F_{i_j} = f_j$, $1 \leq j \leq k$, and
    \item $i \not \in \{i_1, i_2, \ldots, i_k\}$ implies $F_i$ is a
          trivial loop. $\Box$
  \end{itemize}
\end{enumerate}
\end{Definition}

This notion lets us compare the digital homotopy properties of loops
whose domains may have differing cardinality, since if $m_1 \leq m_2$,
we can obtain \cite{Boxer99} a trivial extension of a loop
$f:[0,m_1]_{{\bf Z}} \rightarrow X$ to
$f':[0,m_2]_{{\bf Z}} \rightarrow X$ via
\[ f'(t) =  \left\{
         \begin{array}{ll}
               f(t) & \mbox{if } 0 \leq t \leq m_1; \\
               f(m_1) &  \mbox{if } m_1 \leq t \leq m_2.
         \end{array}
        \right.  \]
Observe that every digital loop $f$ is a trivial extension of itself.

\begin{Definition}
\label{loop-class}
{\rm (\cite{Han}, correcting an earlier definition in \cite{Boxer05a}).}
Two loops $f_0,f_1$ with the same base point $p \in X$
belong to the same loop class~$[f]_X$ if they have trivial extensions
that can be joined by a homotopy $H$ that keeps the endpoints fixed.
\end{Definition}

It was incorrectly asserted as Proposition~3.1 of~\cite{Boxer05a} 
that the assumption in Definition~\ref{loop-class}, that the homotopy keeps the endpoints fixed, 
could be replaced by the weaker assumption
that the homotopy is loop-preserving; the error was pointed out in \cite{Boxer06}.

Membership in the same loop class in $(X,x_0)$ is
an equivalence relation among loops
~\cite{Boxer99}.



The digital fundamental group is derived from
a classical notion of algebraic topology (see \cite{Massey,Munkres,Spanier}).
The version discussed in this section is that developed in~\cite{Boxer99}.
The next result is used in \cite{Boxer99} to show the product operation of our
digital fundamental group is well defined.

\begin{Proposition}
\label{well-defined}
{\rm \cite{Boxer99,Khalimsky}}
Let $f_1,f_2,g_1,g_2$ be digital loops based at $x_0$ in a pointed digital
image $(X,x_0)$, with $f_2 \in [f_1]_X$ and
$g_2 \in [g_1]_X$.  Then $f_2 * g_2 \in [f_1 * g_1]_X$. \qed
\end{Proposition}

Let $(X,x_0)$ be a pointed digital image;
{\em i.e.}, $X$ is a digital image, and $x_0 \in X$.
Define $\Pi_1(X,x_0)$ to be the set of loop classes
$[f]_X$ in $X$ with base point $x_0$. When we wish to emphasize an adjacency relation $\kappa$, we denote this set by $\Pi_1^\kappa(X,x_0)$. 
By Proposition~\ref{well-defined}, the {\em product} operation
\[ [f]_X \cdot [g]_X~=~[f * g]_X \]
is well defined on $\Pi_1(X,x_0)$;
further, the operation $\cdot$ is associative on
$\Pi_1(X,x_0)$~\cite{Khalimsky}.

\begin{Lemma}
\label{ident-elt}
{\rm \cite{Boxer99}}
Let $(X,x_0)$ be a pointed digital image.
Let ${\overline{x_0}}: [0,m]_{{\bf Z}} \rightarrow X$ be a constant loop
with image $\{x_0\}$.  Then $[{\overline{x_0}}]_X$ is an identity element for
$\Pi_1(X, x_0)$. \qed
\end{Lemma}

\begin{Lemma}
\label{inverse}
{\rm \cite{Boxer99}}
If $f: [0,m]_{{\bf Z}} \rightarrow X$ represents an element of $\Pi_1(X,x_0)$,
then the reverse loop $f^{-1}$ 
is an element of $[f]_X^{-1}$ in $\Pi_1(X,x_0)$. \qed
\end{Lemma}

\begin{Theorem}
{\rm \cite{Boxer99}}
$\Pi_1(X,x_0)$ is a group under the $\cdot$ product operation,
the {\em fundamental group of} $(X,x_0)$.
\qed
\end{Theorem}

\begin{Theorem}
\rm{\cite{Boxer99}}
\label{induced}
Suppose $F: (X, \kappa, x_0) \rightarrow (Y,\lambda, y_0)$
is a pointed continuous function. Then
$F$ induces a homomorphism
$F_*: \Pi_1^{\kappa}(X,x_0) \rightarrow
    \Pi_1^{\lambda}(Y,y_0)$
defined by $F_*([f])=[F \circ f]$. $\Box$
\end{Theorem}

\section{Homotopy equivalent images that aren't pointed homotopy equivalent}
In~\cite{Boxer05a}, it was asked if, given digital images $(X,\kappa)$ and
$(Y,\lambda)$ that are homotopy equivalent, must $(X,x_0,\kappa)$ and
$(Y,y_0,\lambda)$ be pointed homotopy equivalent for arbitrary
base points $x_0 \in X$, $y_0 \in Y$? The paper~\cite{Staecker} gives an
example, not using any of the $c_u$-adjacencies,
that answers this question in the negative.
It is desirable to have an example that uses $c_u$-adjacencies. In this
section, we give such an example by modifying that of ~\cite{Staecker}.

\begin{Example}
\label{the-ex}
Let $X = \{x_i\}_{i=0}^{10} \subset {\bf Z}^2$ where
$x_0=(2,0)$, $x_1=(1,1)$, $x_2=(0,2)$, $x_3=(-1,2)$, $x_4=(-2,1)$,
$x_5=(-2,0)$, $x_6=(-2,-1)$, $x_7=(-1,-2)$, $x_8=(0,-2)$, $x_9=(1,-1)$,
$x_{10}=(0,0)$. Let $Y=X \setminus \{x_0\} = \{x_i\}_{i=1}^{10}$. We consider both $X$ and $Y$ as digital images with $c_2$-adjacency.
See Figure~1. $\Box$
\end{Example}

\begin{figure}
\label{2b-htpc-single}

\begin{center}\begin{tikzpicture}[scale=.35]
	\filldraw[fill=white, xshift=2cm,yshift=4cm]
		(45:1.2) \foreach \x in {135,225,315,45} { -- (\x:1.2) };
	\filldraw[fill=white, xshift=2cm,yshift=6cm]
		(45:1.2) \foreach \x in {135,225,315,45} { -- (\x:1.2) };
	\filldraw[fill=white, xshift=2cm,yshift=8cm]
		(45:1.2) \foreach \x in {135,225,315,45} { -- (\x:1.2) };
	\filldraw[fill=white, xshift=4cm,yshift=2cm]
		(45:1.2) \foreach \x in {135,225,315,45} { -- (\x:1.2) };
	\filldraw[fill=white, xshift=4cm,yshift=10cm]
		(45:1.2) \foreach \x in {135,225,315,45} { -- (\x:1.2) };
	\filldraw[fill=white, xshift=6cm,yshift=2cm]
		(45:1.2) \foreach \x in {135,225,315,45} { -- (\x:1.2) };
	\filldraw[fill=white, xshift=6cm,yshift=6cm]
		(45:1.2) \foreach \x in {135,225,315,45} { -- (\x:1.2) };
	\filldraw[fill=white, xshift=6cm,yshift=10cm]
		(45:1.2) \foreach \x in {135,225,315,45} { -- (\x:1.2) };
	\filldraw[fill=white, xshift=8cm,yshift=4cm]
		(45:1.2) \foreach \x in {135,225,315,45} { -- (\x:1.2) };
	\filldraw[fill=white, xshift=8cm,yshift=8cm]
		(45:1.2) \foreach \x in {135,225,315,45} { -- (\x:1.2) };
	\filldraw[fill=white, xshift=10cm,yshift=6cm]
		(45:1.2) \foreach \x in {135,225,315,45} { -- (\x:1.2) };
	\node () at (2cm,4cm) {$x_{6}$};
	\node () at (2cm,6cm) {$x_{5}$};
	\node () at (2cm,8cm) {$x_{4}$};
	\node () at (4cm,2cm) {$x_{7}$};
	\node () at (4cm,10cm) {$x_{3}$};
	\node () at (6cm,2cm) {$x_{8}$};
	\node () at (6cm,6cm) {$x_{10}$};
	\node () at (6cm,10cm) {$x_{2}$};
	\node () at (8cm,4cm) {$x_{9}$};
	\node () at (8cm,8cm) {$x_{1}$};
	\node () at (10cm,6cm) {$x_{0}$};
\end{tikzpicture}
\end{center}
\caption{A figure $X=\{x_i\}_{i=0}^{10}$ and its subset $Y=X\setminus \{x_0\}$
that are homotopic but not pointed homotopic as images in ${\bf Z}^2$ with $c_2$-adjacency}
\end{figure}
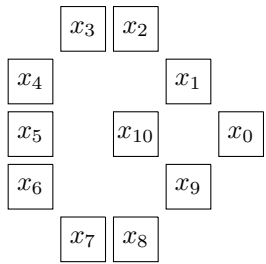

\begin{Proposition}
\label{htpy-equiv}
Let $X$ and $Y$ be the images of Example~\ref{the-ex}.  Then
$X$ and $Y$ are $(c_2,c_2)$-homotopy equivalent.
\end{Proposition}

\begin{proof}
Let $f: X \rightarrow Y$ be defined by
\[ f(x_i) = \left \{ \begin{array}{ll}
            x_{i+1} & \mbox{ if } 0 \leq i \leq 9; \\
            x_1 & \mbox{ if } i=10.
            \end{array}
            \right . \]
Let $g: Y \rightarrow X$ be the inclusion map. Clearly, both $f$ and $g$ are
$(c_2, c_2)$-continuous.
The function $H: X \times [0,1]_Z \rightarrow X$ defined by
\[ H(x_i,t) = \left \{ \begin{array}{ll}
                       f(x_i) = g \circ f(x_i) & \mbox{ if } t=0; \\
                       x_i & \mbox{ if } t=1,
                      \end{array}
                      \right .
\]
is clearly a $(c_2,c_2)$-homotopy between $g \circ f$ and $1_X$.
The function $K: Y \times [0,1]_Z \rightarrow Y$ defined by
\[ K(x_i,t) = \left \{ \begin{array}{ll}
                       f(x_i) = f \circ g(x_i) & \mbox{ if } t=0 
                                \mbox{ and } 1 \leq i \leq 10; \\
                       x_i & \mbox{ if } t=1
                                \mbox{ and } 1 \leq i \leq 10,
                      \end{array}
                      \right .
\]
is clearly a $(c_2,c_2)$-homotopy between $f \circ g$ and $1_Y$.
Thus, $(X,c_2)$ and $(Y,c_2)$ are homotopy equivalent. 
\end{proof}


\begin{Proposition}
\label{Ymust-be-id}
Let $Y=\{x_i\}_{i=1}^{10}$ be as above.
Let $h: (Y, c_2) \rightarrow (Y,c_2)$ be a continuous map such
that $h(x) = x$ for some $x\in Y$ and $h$ is $(c_2,c_2)$-homotopic to $1_Y$ in 1 step.
Then $h=1_Y$.
\end{Proposition}

\begin{proof}
For convenience, we prove the statement in the case where $x=x_1$. Since $(Y,c_2)$ is a simple cycle of 10 points, the same argument will work for any other value of $x$.

Since $h$ is $(c_2,c_2)$-homotopic to $1_Y$ in 1 step,
$h(x_i)$ and $x_i$ are $c_2$-adjacent or equal for all $i$.
Suppose $h \neq 1_Y$. Since $h(x_1)=x_1$,
by $c_2$-continuity, $h(x_i) \in \{x_{i-1},x_i\}$ for $2 \leq i \leq 10$,
and since $h \neq 1_Y$, there is a $j_0$ such that $2 \leq j_0 \leq 10$
and $h(x_j)=x_{j-1}$ for
$j_0 \leq j \leq 10$. In particular, $h(x_{10})=x_9$, so we have a discontinuity
since the $c_2$-adjacent points $x_1$ and $x_{10}$ do not have $c_2$-adjacent
images under $h$. Since $h$ was assumed continuous, the contradiction leads
us to conclude that $h=1_Y$.
\end{proof}

A similar argument shows the following.

\begin{Corollary}
\label{Xmust-be-id}
Let $X=\{x_i\}_{i=0}^{10}$ be as above.
Let $h: (X, c_2) \rightarrow (X,c_2)$ be a continuous map such
that $h(x_0) = x_0$ and $h$ is homotopic in 1 step to $1_X$. Then $h=1_X$. \qed
\end{Corollary}

\begin{Proposition}
\label{proof-of-example}
Let $X= \{x_i\}_{i=0}^{10}$ and $Y=X\setminus \{x_0\}$ be as above.  Then for any
$x \in X$ and $y \in Y$,
$(X,x)$ and $(Y, y)$ are not pointed $(c_2,c_2)$-homotopy equivalent.
\end{Proposition}

\begin{proof}
Suppose otherwise. Then for some $x \in X$ and $y \in Y$, there are
$(c_2,c_2)$-continuous pointed
maps $f: (X,x) \rightarrow (Y, y)$ and $g: (Y,y) \rightarrow (X, x)$
such that $f \circ g$ is pointed homotopic to $1_X$ and
$g \circ f$ is pointed homotopic to $1_Y$.  

First we argue that $g\circ f$ must in fact equal $1_X$. Since $f$ and $g$ are pointed maps we have $g\circ f(x) = x$, and our pointed homotopy from $g\circ f$ to $1_X$ will fix $x$ at all stages. If $g\circ f$ were not $1_X$, then there would be some final stage $h$ of the pointed homotopy from $g\circ f$ to $1_X$ for which $h \neq 1_X$ but $h$ is pointed homotopic to $1_X$ in one step. This is impossible by Proposition \ref{Ymust-be-id}, and so we conclude that $g\circ f = 1_X$. Similarly, using Corollary \ref{Xmust-be-id}, we have $f\circ g = 1_Y$. 

Since $f \circ g = 1_Y$ and $g\circ f = 1_X$, it follows that $X$ and $Y$ are $(c_2,c_2)$-isomorphic images, which is impossible, as $X$ and $Y$ have different cardinalities. The assertion
follows. 
\end{proof}

Example \ref{the-ex} is an image in $\mathbf Z^2$ with $c_2$-adjacency that exhibits interesting pointed homotopy properties. We remark that images exist in $\mathbf Z^2$ with $c_1$-adjacency with similar properties. The image in Figure \ref{4adj-ex} exhibits the same behavior as that of Example \ref{the-ex}. 

\begin{figure}
\begin{center}
\begin{tikzpicture}[scale=.3]
	\filldraw[fill=white, xshift=2cm,yshift=2cm]
		(45:1.2) \foreach \x in {135,225,315,45} { -- (\x:1.2) };
	\filldraw[fill=white, xshift=2cm,yshift=4cm]
		(45:1.2) \foreach \x in {135,225,315,45} { -- (\x:1.2) };
	\filldraw[fill=white, xshift=2cm,yshift=6cm]
		(45:1.2) \foreach \x in {135,225,315,45} { -- (\x:1.2) };
	\filldraw[fill=white, xshift=2cm,yshift=8cm]
		(45:1.2) \foreach \x in {135,225,315,45} { -- (\x:1.2) };
	\filldraw[fill=white, xshift=4cm,yshift=2cm]
		(45:1.2) \foreach \x in {135,225,315,45} { -- (\x:1.2) };
	\filldraw[fill=white, xshift=4cm,yshift=8cm]
		(45:1.2) \foreach \x in {135,225,315,45} { -- (\x:1.2) };
	\filldraw[fill=white, xshift=6cm,yshift=2cm]
		(45:1.2) \foreach \x in {135,225,315,45} { -- (\x:1.2) };
	\filldraw[fill=white, xshift=6cm,yshift=6cm]
		(45:1.2) \foreach \x in {135,225,315,45} { -- (\x:1.2) };
	\filldraw[fill=white, xshift=6cm,yshift=8cm]
		(45:1.2) \foreach \x in {135,225,315,45} { -- (\x:1.2) };
	\filldraw[fill=white, xshift=8cm,yshift=2cm]
		(45:1.2) \foreach \x in {135,225,315,45} { -- (\x:1.2) };
	\filldraw[fill=white, xshift=8cm,yshift=4cm]
		(45:1.2) \foreach \x in {135,225,315,45} { -- (\x:1.2) };
	\filldraw[fill=white, xshift=8cm,yshift=6cm]
		(45:1.2) \foreach \x in {135,225,315,45} { -- (\x:1.2) };
	\filldraw[fill=white, xshift=8cm,yshift=8cm]
		(45:1.2) \foreach \x in {135,225,315,45} { -- (\x:1.2) };
\end{tikzpicture}
\end{center}
\caption{An image in $Z^2$ with $c_1$-adjacency having the same properties as in Example \ref{the-ex}.\label{4adj-ex}}
\end{figure}
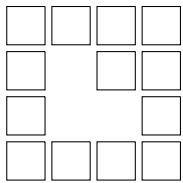

Let $X$ be the digital image in Example \ref{the-ex}, and define two loops $f,g: [0,10]_Z \to X$ as follows:
\begin{align*}
f &= (x_1, x_2, \dots, x_9, x_{10}, x_1) \\
g &= (x_1, x_2, \dots, x_9, x_0, x_1) 
\end{align*}

These loops are equivalent in $\Pi_1(X,x_1)$: consider the following trivial extensions
\begin{align*}
f' &= (x_1, x_2, x_3, \dots, x_9, x_{10}, x_1, x_1) \\
g' &= (x_1, x_1, x_2, \dots, x_8, x_9, x_0, x_1)
\end{align*}
These loops $f'$ and $g'$ are homotopic in one step, and so $f$ and $g$ are equivalent in $\Pi_1(X,x_1)$. Notice that the one-step equivalence above uses trivial extensions at the base point $x_1$. That is, there is some $t$ with $f'(t)=f'(t+1) = x_1$, and likewise for $g'$. In fact this is necessary for any equivalence between $f$ and $g$, as the following proposition shows:
\begin{Proposition}
Let $X$ be as in Example~\ref{the-ex}. Let $f$ and $g$ be the loops described above.
Let $f',g':[0,k]_Z \to X$ be trivial extensions of $f$ and $g$ that are homotopic by $H(t,s):[0,k]_Z \times [0,n]_Z \to X$. Then there is some time $p \in [0,n]_Z$ 
and intermediate stage of the homotopy $H$, i.e., 
$h: [0,k]_{\Z} \rightarrow X$ defined by
$h(t) = H(t,p)$,  such that $h(k-1)=h(k)= x_1$.
Similarly there is some $q \in [0,n]_Z$ and intermediate stage of the
homotopy $H$, i.e., $l: [0,k]_{\Z} \rightarrow X$ defined by $l(t)=H(t,q)$,
such that $l(0)=l(1)=x_1$. 
\end{Proposition}

\begin{proof}
We will prove the first statement; the second follows similarly. Suppose
that no intermediate loop $h$ obeys $h(k-1)=h(k)=x_1$. 
Then we have $H(k-1,s) \neq x_1$ for all $s$. We must in particular have
$f'(k-1) \neq x_1$, and so $f'(k-1) = x_{10}$ since $f'$ is a trivial extension
of $f$. 

Thus, considering $H(k-1,s)$ for various $s$ gives a path from 
$H(k-1,0) = f'(k-1) = x_{10}$ to $g'(k-1) = x_0$ which never 
passes through $x_1$. Because of the structure of our image $X$, 
this path must at some point pass through $x_9$. 
Thus there is some $r$ with $H(k-1,r) = x_9$. But $H(k,r) = x_1$
since all stages of $H$ are loops at $x_1$. This contradicts continuity 
of $H$ from $H(k-1,r)$ to $H(k,r)$ since $x_9$ is not adjacent to $x_1$ in $X$. 
\end{proof}

Thus we see that $f$ and $g$ are equivalent as loops in $\Pi_1(X,x_1)$, but this equivalence requires trivial extensions at the base point. This suggests a finer equivalence relation than the one used for the fundamental group, one in which loops are equivalent only by homotopies that do not extend the base point. Specifically, we call a loop $f$ \emph{tight at the basepoint (TAB)} $x_0$ when there is no $t$ with $f(t)=f(t+1)=x_0$. Two TAB loops are called TAB equivalent when there are TAB trivial extensions that are homotopic by a homotopy which is TAB in each stage.

Thus our example loops $f$ and $g$ above are equivalent in $\Pi_1(X,x_1)$, but not TAB equivalent, because any homotopy of trivial extensions must have a non-TAB intermediate stage. The equivalence classes using the TAB relation seem to have interesting and subtle structure, but they do not naturally form a group with respect to the product operation, as we show below.

Consider the product of $f$ and the reverse of $g$, which has the form:
\[ f*g^{-1} = (x_1,x_2, \dots, x_9,x_{10},x_1,x_0,x_9,\dots,x_2,x_1) \]
Note that $f*g^{-1}$ is nullhomotopic, using only TAB loops as intermediate steps. The first step of the nullhomotopy is as follows:
\begin{align*}
 (x_1,x_2, \dots, x_9,x_{10},x_1,x_0,x_9,\dots,x_2,x_1) \mbox{ to}\\
 (x_1,x_2, \dots, x_9,x_9,x_0,x_0,x_9,\dots,x_2,x_1),
\end{align*}
and then the loop deforms continuously to the constant map
$(x_1,x_1, \ldots, x_1)$ in an obvious way.

Since $f$ and $g$ are not TAB equivalent, but $f*g^{-1}$ is pointed nullhomotopic, the TAB relation, which is finer than the equivalence used in $\Pi_1(X,x_1)$, cannot be used to define a group. Nevertheless the TAB equivalence provides subtle and interesting information about loops in our space.

\section{A new formulation of the fundamental group}
\label{ec-section}
The equivalence relation of Definition~\ref{loop-class} used to define the fundamental group relies on trivial extensions, which are often cumbersome to handle. In this section we give an equivalent definition of the fundamental group which does not require trivial extensions. Our construction instead is based on \emph{eventually constant paths}. Let $\N= \{1,2,\dots\}$ denote the natural numbers,
and $\N^* = \{0\} \cup \N$. We consider $\N^*$ to be a digital image with 2-adjacency.

\begin{Definition}
Given a digital image $X$, a continuous function $f:\N^* \to X$ is called an \emph{eventually constant path} or \emph{EC path} if there is some point $c\in X$ and some $N \geq 0$ such that $f(x)=c$ whenever $x\geq N$. When convenient we abbreviate the latter by $f(\infty) = c$. The \emph{endpoints} of an EC path $f$ are the two points $f(0)$ and $f(\infty)$. 

If $f$ is an EC path and $f(0)=f(\infty)$, we say $f$ is an \emph{EC loop}, and $f(0)$ is called the basepoint.

We say that a homotopy $H$ between EC paths is an \emph{EC homotopy} when the function $H_t: \N^* \rightarrow X$ defined by
$H_t(s) = H(s,t)$ is an EC path for all $t \in [0,k]_{\Z}$. To indicate an EC homotopy, we write $f \echtp g$, or $f \echtp_{\kappa} g$ if it is desirable to
state the adjacency $\kappa$ of $X$.
We say an EC homotopy $H$ \emph{holds the endpoints fixed} when $H_t(0) = f(0) = g(0)$ and there is a $c \in \N^*$ such that
$n \geq c$ implies $H_t(n) = f(n) = g(n)$ for all $t$. 
$\Box$
\end{Definition}

Not all homotopies of EC paths are EC homotopies, as the following example shows.

\begin{Example}
\label{wobbler}
Let $f,g: \N^* \rightarrow [0,1]_{\Z}$ be defined by
$f(0)=g(0)=0$, $f(n)=g(n)=1$ for $n>0$.
Let $H: \N^* \times [0,2]_{\Z} \rightarrow [0,1]_{\Z}$ be defined by
$H_0=H_2=f=g$,
$ H_1(s)=0$ if $s$ is even, $H_1(s)= 1$ if $s$ is odd. Then $H$ is a homotopy 
from $f$ to $g$ that is not an EC homotopy.
\end{Example}

\begin{proof} It is easy to see that $H$ is a homotopy. However, $H_1$ 
is not an EC path.  The assertion follows.
\end{proof}

A familiar argument shows that EC homotopy is an equivalence relation.

\begin{Proposition}
\label{EC-equiv-relation}
EC homotopy and EC homotopy holding the endpoints fixed are equivalence relations among EC paths.
\end{Proposition}

\begin{proof} We give a proof without the assumption of endpoints being
held fixed.  The same argument can be used with obvious modifications to
obtain the assertion for endpoints held fixed.

{\em Reflexive}: Given an EC path $f: \N^* \rightarrow X$, clearly the function
$H: \N^* \times \{ 0 \} \rightarrow X$ given by
$H(x,0) = f(x)$ shows $f \echtp f$.

{\em Symmetric}: If $H: \N^* \times [0,m]_{\Z} \rightarrow X$ is an EC homotopy from
$f$ to $g$, then it is easy to see that the function
$H': \N^* \times [0,m]_{\Z} \rightarrow X$ defined by
\[ H'(x,t) = \left \{ \begin{array}{ll}
                      H(x,m-t) & \mbox{if } 0 \leq t \leq m; \\
                      f(0) & \mbox{if } t \geq m,
                      \end{array} \right .
 \]
 shows $g \echtp f$.

{\em Transitive}: Suppose $H: \N^* \times [0,m_1]_{\Z} \rightarrow X$ is 
an EC homotopy from $f$ to $g$, and $K : \N^* \times [0,m_2]_{\Z} \rightarrow X$
is an EC homotopy from $g$ to $h$.  Then the function
$L: \N^* \times [0,m_1+m_2]_{\Z} \rightarrow X$ defined by
\[ L(x,t) = \left \{ \begin{array}{ll}
                   H(x,t) & \mbox{if } 0 \leq t \leq m_1; \\
                   K(x,t-m_1) & \mbox{if } m_1 \leq t \leq m_2,
                   \end{array} \right .
\]
is an EC homotopy from $f$ to $h$.
\end{proof}

Homotopy of trivial extensions of loops can be easily stated in terms of EC homotopy of the corresponding EC loops. The latter formulation is preferable since it does not require trivial extensions, which obviates the need for several technical lemmas. For example the proof given below for Proposition \ref{EC-mult-well-defined} is much easier than the corresponding statement for trivial extensions (see \cite[Proposition 4.8]{Boxer}, which is only a sketch of a proof from \cite{Khalimsky}); and the proof given below for
Theorem~\ref{Pi1-iso-thm} is somewhat simpler, being
based on EC homotopy, than it would have been had we had to
construct trivial extensions.

Given a path $f: [0,m]_{\Z} \rightarrow X$, we denote by
$f_{\infty}: \N^* \rightarrow X$ the function defined by
\[ f_{\infty}(n) = \left \{ \begin{array}{ll}
             f(n) & \mbox{if } 0 \leq n \leq m; \\
             f(m) & \mbox{if } n \geq m.
      \end{array} \right .
\]

Given an EC path $g: \N^* \rightarrow X$, let 
\[ N_g = \min\{ m \in \N^* \, | \, n \geq m \mbox{ implies } g(n)=g(m) \} \] 
and let $g_-: [0,N_g]_{\Z} = g|_{[0,N_g]_{\Z}}$. We have the following.

\begin{Proposition}
\label{ec-extension-restriction}
Let $X$ be a digital image.

\rm{a)} Let $f: \N^* \rightarrow X$ be an EC path. Then $(f_-)_{\infty}=f$.

\rm{b)} Let $f: [0,m]_{\Z} \rightarrow X$ be a path in $X$. Then $f$ is a trivial extension of $(f_{\infty})_-$. We have $f=(f_{\infty})_-$ if and only if either $m=0$ or $m>0$ and $f(m-1)\neq f(m)$. 
\end{Proposition}

\begin{proof} These assertions are immediate consequences of the
definitions above.
\end{proof}

\begin{Lemma}\label{ec-homotopy}
Let $f,g:[0,m]_\Z \to X$ be paths with $f\simeq g$. Then $f_\infty \echtp g_\infty$. If the homotopy from $f$ to $g$ holds the endpoints fixed, then so does the induced EC homotopy from $f_\infty$ to $g_\infty$.
\end{Lemma}

\begin{proof}
Let $H:[0,m]_\Z \times [0,k]_\Z \to X$ be a homotopy of $f$ to $g$. Consider $G: \N^* \times [0,k]_\Z \to X$, defined as follows:
\[
G(s,t) = \begin{cases}
H(s,t) & \text{ if } s \le m \\
H(m,t) & \text{ if } s > m.
\end{cases} 
\]
Clearly $G$ is an EC homotopy of $f_\infty$ to $g_\infty$. Further, $G$ holds the
endpoints fixed if $H$ does so.
\end{proof}

\begin{Lemma}
\label{cutoff}
Let $f$ and $g$ be EC homotopic EC paths in $X$.  
Then $f_-$ and $g_-$ have homotopic trivial extensions.
If $f$ and $g$ are homotopic holding the endpoints fixed,
then $f_-$ and $g_-$ have trivial extensions that are homotopic holding
the endpoints fixed.
\end{Lemma}

\begin{proof} Let $N_f, N_g$ be as defined above. Without loss of generality,
$N_f \leq N_g$. Let $H: \N \times [0,m]_{\Z} \rightarrow X$ be a homotopy from
$f$ to $g$. Let $H': [0,N_g] \times [0,m]_{\Z} \rightarrow X$ be the restriction of $H$ to
$[0,N_g] \times [0,m]_{\Z}$. It is easily seen that $H'$ is a homotopy between a
trivial restriction $f'$ of $f_-$ and the function $g_-$, where
$f': [0,N_g]_{\Z} \rightarrow X$ is defined by
\[ f'(n) = \left \{ \begin{array}{ll}
                  f(n)=f_-(n) & \mbox{if } 0 \leq n \leq N_f; \\
                  f(N_f) & \mbox{if } N_f  \leq n \leq N_g.
                  \end{array} \right .
\]
Further, if $H$ holds the
endpoints fixed, then so does $H'$.
\end{proof}

\begin{Lemma}\label{te-homotopy}
Let $f:[0,m]_\Z\to X$ be a loop and $\bar f:[0,n]_\Z \to X$ be a trivial extension of $f$. Then $f_\infty$ and $\bar f_\infty$ are EC homotopic with fixed endpoints.
\end{Lemma}

\begin{proof}
We will prove the Lemma in the case that $\bar f$ is obtained from $f$ by inserting a single trivial loop. The full result follows by induction. Specifically, let $f= f_1 * f_2$ and $\bar f = f_1 * c * f_2$, where $c$ is a trivial loop. Say that $f_1:[0,m]_\Z \to X$ and $f_2: [0,n]_\Z \to X$ and $c:[0,k]_\Z \to X$. Then consider  $H: \N^* \times [0,k]_\Z \to X$ given by:
\[ H(s,t) = \begin{cases} 
f_1(s) & \text{ if } 0 \leq s \leq m; \\
c(s-m) & \text{ if } m \le s \le m+t; \\
f_2(s-(m+t)) & \text { if } m+t \le s \le m+t+n; \\
x_0 & \text{ if } m+t+n \le s.
\end{cases}
\]
At time stage $t$ we have
$H_t = (f_1 * c_{|_{[0,t]_{\Z}}} * f_2)_{\infty}$, so
$H$ is an EC homotopy of $f_\infty$ to $\bar f_\infty$ as desired. Further,
$H$ fixes the endpoints, since $H(0,t)=f_1(0)$ for all $t$ and $H(x,t)=f_2(n)$
for all $x \geq m+t+n$ and all $t$.
\end{proof}

\begin{Theorem}\label{ec-te-equiv}
Let $f$ and $g$ be loops in $X$ having some common basepoint $p$. Then there are trivial extensions $\bar f, \bar g$ of $f, g$ respectively with $\bar f \simeq \bar g$ with fixed endpoints if and only if $f_\infty$ and $g_\infty$ are EC homotopic with fixed endpoints.
\end{Theorem}
\begin{proof}
First we assume that there are trivial extensions $\bar f,\bar g$ with $\bar f\simeq \bar g$ fixing endpoints. Then by Lemmas \ref{te-homotopy} and \ref{ec-homotopy} we have $f_\infty \echtp \bar f_\infty \echtp \bar g_\infty \echtp g_\infty$ and all homotopies fix the endpoints as desired.

For the converse assume that $f_\infty \echtp g_\infty$ with fixed endpoints. Let $H:\N^* \times [0,k]_\Z \to X$ be the EC homotopy. Since $H$ fixes the endpoints (at $p$) and has only finitely many stages, there must be some $M$ such that $H(s,t) = p$ for all $s\ge M$ and for all $t$. 

Let $\bar f,\bar g:[0,M]_\Z \to X$ be the restrictions of $f_\infty,g_\infty$ respectively to $[0,M]_\Z$. Then $\bar f = f*c$ is a trivial extension of $f$, where $c$ is a trivial loop at $p$. Similarly $\bar g$ is a trivial extension of $g$.

Let $\bar H:[0,M]_\Z \times [0,k]_\Z \to X$ be the restriction of $H$ to $[0,M]_\Z \times [0,k]_\Z$. Then $H$ is a homotopy of $\bar f$ to $\bar g$ fixing the endpoints as desired.
\end{proof}

It is natural to overload the $*$ notation as follows.
\begin{Definition}
\label{EC-star}
For $x_0 \in X$, let $f_0, f_1: \N^* \rightarrow X$ be $x_0$-based EC loops in $X$.  
 Define $f_0 * f_1: \N^* \rightarrow X$ by
\[ f_0 * f_1(n) = \left \{ \begin{array}{ll}
                                f_0(n) & \mbox{if } 0 \leq n \leq N_{f_0}; \\
                                f_1(n-N_{f_0}) & \mbox{if } N_{f_0} \leq n.~~~~~ \Box
                                \end{array} \right .
\]
\end{Definition}

It is easily seen that $f_0 * f_1$ is well defined and is an EC loop in $X$. 
The $*$ operator on EC loops has the following properties.

\begin{Proposition}
\label{star-props}
\begin{itemize}
\item Let $f,g: \N^* \rightarrow X$ be $x_0$-based EC loops, for some $x_0 \in X$. Then $f_-*g_-=(f*g)_-$.
\item Let $f: [0,m]_{\Z} \rightarrow X$, $g: [0,n]_{\Z} \rightarrow X$ be $x_0$-based EC loops, for some $x_0 \in X$. Then $f_{\infty} * g_{\infty} = (f*g)_{\infty}$.
\end{itemize}
\end{Proposition}

\begin{proof} These properties are simple consequences of
Definition~\ref{EC-star}.
\end{proof}

\begin{Lemma}\label{EC-mult-half}
Let $f, g, g'$ be EC loops in $X$ at a common basepoint, with $g\echtp g'$ holding the endpoints fixed. Then $f * g \echtp f * g'$ holding the endpoints fixed.
\end{Lemma}
\begin{proof}
Let $H:\N^* \times [0,m] \to X$ be the EC homotopy from $g$ to $g'$, and let $L:\N^* \times [0,m] \to X$ be given by
\[ L(s,t) = (f * H_t)(s). \]
Then $L$ is a EC homotopy from $f*g$ to $f*g'$ holding the endpoints fixed as desired.
\end{proof}

In order to prove Proposition~\ref{EC-mult-well-defined} below,
we must take care in how we mimic the proof of Lemma~\ref{EC-mult-half}
on the first factors of the * products, as shown by the following.

\begin{Example}
\label{htpy-discont-in-t}
Let $f,g: \N^* \rightarrow [0,1]_{\Z}$ be defined by
\[ f(n)=g(n)= \left \{ \begin{array}{ll}
        n & \mbox{if } n \in \{0,1,2\}; \\
        1 & \mbox{if } n=3; \\
         0 & \mbox{if } n > 3.
         \end{array} \right .
\]
Then there is an EC homotopy
$H: \N^* \times [0,2]_{\Z} \rightarrow [0,1]_{\Z}$ from $f$ to $f$ such that the function
$K: \N^* \times [0,2]_{\Z} \rightarrow [0,1]_{\Z}$ defined by
$K(n,t) = H_t(n) * g(n)$ is not continuous in $t$, where
$H_t: \N^* \rightarrow [0,1]_{\Z}$ is the induced function $H_t(n)=H(n,t)$.
\end{Example}

\begin{proof} Define $H(n,t)$ by $H(n,0)=f(n)=g(n)=H(n,2)$,
\[ H(n,1)= \left \{ \begin{array}{ll}
                   f(n) & \mbox{if } n \neq 5; \\
                   1      & \mbox{if } n=5.
                   \end{array} \right .
\]
It is easy to see that $H$ is a homotopy.  However,
$K=H_0 * g = H_2 * g$ and $L=H_1*g$ are represented respectively by the sequences
\[ (K(0),K(1),K(2),\ldots)=(0,1,2,1,0,1,2,1,0,0,\ldots) \]
\[ (L(0),L(1),L(2),\ldots)=(0,1,2,1,0,1,0,1,2,1,0,0,\ldots) \]
In particular, $H_0*g(6)=2$ and $H_1 * g(6)=0$, so at $n=6$,
$H_t*g$ is not continuous in $t$.
\end{proof}

\begin{Proposition}
\label{EC-mult-well-defined}
Let $f,f',g,g'$ be EC loops in $X$ at a common basepoint such that
$f\echtp f'$ and $g \echtp g'$ with both homotopies
holding the endpoints fixed.
Then we have $f*g\echtp f' * g'$ holding the endpoints fixed.
\end{Proposition}
\begin{proof}
By Lemma \ref{EC-mult-half} we have $f*g \echtp f*g'$ holding the endpoints fixed.

By an argument similar to that of the proof of Lemma \ref{EC-mult-half} we will show that $f*g' \echtp f' * g'$.
Example~\ref{htpy-discont-in-t} shows that
$H_t*g'$ will not necessarily be continuous in $t$; however, this is easily fixed by inserting an extra constant segment in the first factor. In particular, let $H: \N^* \times [0,m]_{\Z} \rightarrow X$ be an
EC homotopy from $f$ to $f'$ that holds the endpoints fixed. Let
$M = \max\{N_{H_t} \, | \, t \in [0,m]_{\Z}\}$. For each $t \in [0,m]_{\Z}$,
let $c_t : [0, M-N_{H_t}]_{\Z} \rightarrow \{x_0\}$ be a constant function.
Then the function $K: \N^* \times [0,m]_{\Z} \rightarrow X$
defined by
$K(n,t) = (H_t * c_t * g')(n)$ 
is an EC homotopy from $f*g'$ to $f' * g'$ that holds the endpoints
fixed.

Thus by transitivity of EC homotopy we have $f*g\echtp f'*g'$,
holding endpoints fixed.
\end{proof}

Let $G(X,x_0)$ be the set of all EC homotopy classes of EC loops in $X$ based at $x_0$. 

\begin{Proposition}
\label{EC-grp}
$G(X,x_0)$ with the $\cdot$ operation defined by
$[f] \cdot [g] = [f*g]$ is a group.
\end{Proposition}

\begin{proof} By Proposition~\ref{EC-mult-well-defined}, the $\cdot$ operation
is closed and well defined on $G(X,x_0)$. Clearly, the EC pointed
homotopy class of the constant map $c(n)=x_0$ for all $n \in \N$ is the
identity element. Given an $x_0$-based EC loop $f: \N^* \rightarrow X$, the function
$g: \N^* \rightarrow X$ defined by 
\[ g(n)= \left \{ \begin{array}{ll}
                 f(N_f -n) & \mbox{if } 0 \leq n \leq N_f; \\
                 x_0 & \mbox{if } n \geq N_f,
                 \end{array} \right .
\]
gives an inverse for $[f]$.
\end{proof}

We have the following analog of Theorem~\ref{induced}.

\begin{Theorem}
\label{ec-induced}
Suppose $F: (X, \kappa, x_0) \rightarrow (Y,\lambda, y_0)$
is a pointed continuous function. Then
$F$ induces a homomorphism
$F_*: G(X,x_0) \rightarrow
    G(Y,y_0)$
defined by $F_*([f])=[F \circ f]$.
\end{Theorem}

\begin{proof} Given $x_0$-based EC loops $f,g: \N \rightarrow X$, we have, by 
using Propositions~\ref{ec-extension-restriction} and ~\ref{star-props},
\begin{align*}
F([f*g])&=[F\circ (f*g)]=  [F \circ ((f*g)_-)_{\infty} ] 
= [((F \circ f_-) * (F \circ g_-))_{\infty}] \\
&= [(F \circ f_-)_{\infty} * (F \circ g_-)_{\infty}] 
= [(F\circ f) * (F \circ g)].
\end{align*}
The assertion follows.
\end{proof}

The main result of this section is the following.

\begin{Theorem}
\label{iso-groups}
Given a digital image $X$ and a point $x_0 \in X$, the groups $G(X,x_0)$ and
$\Pi_1(X,x_0)$ are isomorphic.
\end{Theorem}

\begin{proof} Let $F: \Pi_1(X,x_0) \rightarrow G(X,x_0)$ be defined by
$F([f]_X)=[f_{\infty}]_X$, where $[f_{\infty}]_X$ is the set of EC loops that are $x_0$-based
in $X$ and are EC homotopic in $X$ to $f_{\infty}$ holding the endpoints fixed.

From Lemma~\ref{cutoff}, $F$ is one-to-one. Also, $F$ is onto, since
given an $x_0$-based EC loop $f$, we have $[f]=F([f_-])$.
From Proposition~\ref{EC-mult-well-defined}, $F$ is a homomorphism.
The assertion follows.
\end{proof}

\section{Homotopy equivalence and fundamental groups}
In the paper \cite{Boxer05a}, it is asserted that digital images that are
(unpointed) homotopy equivalent have isomorphic fundamental groups.  However,
the proof of this assertion is
incorrect. Roughly, the flaw in the argument given in~\cite{Boxer05a} is
that insufficient care was given to making sure that a certain 
homotopy between two loops holds the endpoints fixed.
In this section, we give a correction.

\begin{Theorem}
\label{change-basepoints}
{\rm \cite{Boxer99}}
Let $(X,\kappa)$ be a digital image and let $p,r$ be points of the same
$\kappa$-component
of $X$. Let $q$ be a $\kappa$-path in $X$ from $p$ to $r$.  Then the induced function
$q_{\#}: \Pi_1^{\kappa}(X,p) \rightarrow \Pi_1^{\kappa}(X,r)$ defined by
$q_{\#}([f]) = [q^{-1}*f*q]$ is an isomorphism.
\qed
\end{Theorem}

Theorem~\ref{change-basepoints} was proven in ~\cite{Boxer99} for
 the version of the fundamental
group based on finite loops. However, essentially the same argument
makes Theorem~\ref{change-basepoints} valid for the version  of the
fundamental group based on EC loops, stated below.

\begin{Corollary}
\label{EC-change-basepoints}
Let $(X,\kappa)$ be a digital image and let $p,r$ be points of the same
$\kappa$-component
of $X$. Let $q$ be a $\kappa$-path in $X$ from $p$ to $r$.  Then the induced function
$q_{\#}: \Pi_1^{\kappa}(X,p) \rightarrow \Pi_1^{\kappa}(X,r)$ defined for a $p$-based EC loop $f$
in $X$ by
$q_{\#}([f]) = [(q^{-1})_{\infty}*f*q_{\infty}]$, is an isomorphism.
\qed
\end{Corollary}

\begin{Theorem}
\label{Pi1-iso-thm}
Suppose $(X,\kappa)$ and $(Y,\lambda)$ are (not necessarily pointed) homotopy equivalent
digital images. Let $F: X \rightarrow Y$, $G: Y \rightarrow X$ be homotopy inverses.
Let $p \in X$. Then
$\Pi_1^{\kappa}(X,p)$ and $\Pi_1^{\lambda}(Y, F(p))$ are isomorphic groups.
\end{Theorem}

\begin{proof}
Let $F_{*}: \Pi_1^{\kappa}(X,p) \rightarrow \Pi_1^{\lambda}(Y,F(p))$ be 
the homomorphism induced by $F$ 
according to Theorem~\ref{ec-induced}.
Let $r=(G \circ F)(p)$.
Let $G_{*}: \Pi_1^{\lambda}(Y,F(p)) \rightarrow \Pi_1^{\kappa}(X,r)$ be 
the homomorphism induced by $G$
according to Theorem~\ref{ec-induced}.
Let $H: X \times [0,m]_Z \rightarrow X$ be a homotopy from $1_X$
to   $G \circ F$.
Let $q$ be the path from $p$ to $r$  defined by
$q(t)=H(p,t)$. 

For $s \in [0,m]_{\bf Z}$, let $q_s: [0,m]_{\bf Z} \rightarrow X$ be the path
from $q(0)=p$ to $q(s)=H(p,s)$ given by $ q_s(t) =  q(\min\{s,t\})$.
For a $p$-based EC loop $f$ in $X$, 
let $K: \N^* \times [0,m]_{\Z} \rightarrow X$ be defined by
\[ K(n,t)=(q_t * (H_t \circ f_-) * (q_t)^{-1})_{\infty}(n).\]

Since $q_t$ is a path from $r$ to 
$q(t) = H(p,t)=H_t(f(0)) = H_t(f_-(N_{f})) = (q_t)^{-1}(0)$, 
$K$ is well defined and, for each $t$, the induced function 
$K_t$ is a EC loop based at $p$.  Also, if we let
 $\overline{p}$ denote the constant EC loop at $p$, then
 \[ K(n,0)=((q_0) * (H_0 \circ f_-) * (q_0)^{-1})_{\infty}(n) = \\
       (\overline{p} * f_- * \overline{p})_{\infty}(n) =f(n) \]
and
\[ K(n,m) = (q_m * (H_m \circ f_-) * (q_m)^{-1}))_{\infty}(n) = \\
    (q * (G \circ F \circ f_-) * q^{-1})_{\infty}(n).
\]
Therefore, $K$ is a EC homotopy from $f$ to
\[ (q * (G \circ F \circ f_-) * q^{-1})_{\infty} =
     q_{\infty} * (G \circ F \circ f_-)_{\infty} * (q^{-1})_{\infty}
     = \\
     q_{\infty} * (G \circ F \circ f) * (q_{\infty})^{-1}\]
that keeps the endpoints fixed.

Let $q_{\#}: \Pi_1^{\kappa}(X,p) \rightarrow \Pi_1^{\kappa}(X,r)$ be defined by
$q_{\#}([f]) = [q_{\infty}*f*(q_{\infty})^{-1}]$. 
By the conclusion of the previous paragraph, the function
$q_{\#} \circ G_* \circ F_*$ is the identity map on $\Pi_1^{\kappa}(X,p)$. We know
from Corollary~\ref{EC-change-basepoints} that $q_{\#}$ is an isomorphism.
It follows that $F_*$ is onto and $G_*$ is one-to-one. A similar argument
shows that $G_*$ is onto and $F_*$ is one-to-one. Therefore,
$F_*$ is an isomorphism.
\end{proof}

\section{Further remarks}
We have given the first example of two digital images with $c_u$-adjacencies
that are homotopy equivalent but not pointed homotopy equivalent. We have introduced a variant of the loop equivalence, based on the notion of
tight at the basepoint (TAB) pointed homotopy, and have explored properties of this notion.
We have given an alternate but equivalent approach to the digital
fundamental group based on EC loops that offers the advantage of
avoiding the often-clumsy use of trivial extensions.
We have
provided a correction to the faulty proof of ~\cite{Boxer05a} that
(unpointed) homotopy equivalent digital images have isomorphic 
fundamental groups.

\end{document}